\documentclass{amsart}
\usepackage{graphics,amsmath,amsthm,amssymb}
\usepackage{color}
\newtheorem{theorem}{Theorem}[section]

\newtheorem{corollary}[theorem]{Corollary}

\newtheorem{remark}[theorem]{Remark}
\newtheorem{lemma}[theorem]{Lemma}

\begin{document}

\title[Arithmetic Properties of Singular Overpartitions]{Arithmetic Properties of Andrews' Singular Overpartitions}
\author[S.-C. Chen]{Shi-Chao Chen}
\address{Institute of Contemporary Mathematics, Department of Mathematics
and Information Sciences, Henan University, Kaifeng, 475001,
China, schen@henu.edu.cn}
\author[M. D. Hirschhorn]{Michael D. Hirschhorn}
\address{School of Mathematics and Statistics, UNSW, Sydney 2052, Australia, \hfill \break m.hirschhorn@unsw.edu.au}
\author[J. A. Sellers]{James A. Sellers}
\address{Department of Mathematics, Penn State University, University Park, PA  16802, USA, sellersj@psu.edu}
\thanks{The first author was supported by the NSF of China (No.11101123).}

\date{\today}

\begin{abstract}
In a very recent work, G. E. Andrews defined the combinatorial objects which he called {\it singular overpartitions} with the goal of presenting a general theorem for overpartitions which is analogous to theorems of Rogers--Ramanujan type for ordinary partitions with restricted successive ranks.  As a small part of his work, Andrews noted two congruences modulo 3 which followed from elementary generating function manipulations.  In this work, we show that Andrews' results modulo 3 are two examples of an infinite family of congruences modulo 3 which hold for that particular function.  We also expand the consideration of such arithmetic properties to other functions which are part of Andrews' framework for singular overpartitions. 
\end{abstract}

\maketitle


\bigskip


\section{Introduction}
\label{intro}
In a very recent work, Andrews \cite{A_sing} defined the combinatorial objects which he called {\it singular overpartitions} with the goal of presenting a general theorem for overpartitions which is analogous to theorems of Rogers--Ramanujan type for ordinary partitions with restricted successive ranks.   In the process, Andrews proves that these singular overpartitions, which depend on two parameters $k$ and $i,$ can be enumerated by the function $\overline{C}_{k,i}(n)$ which gives the number of overpartitions of $n$ in which no part is divisible by $k$ and only parts $\equiv \pm i \pmod{k}$ may be overlined.  Andrews also notes that, for all $n\geq 0,$ $\overline{C}_{3,1}(n) = \overline{A}_3(n)$ where $\overline{A}_3(n)$ is the number of overpartitions of $n$ into parts not divisible by 3.  The function $\overline{A}_k(n),$ which counts the number of overpartitions of $n$ into parts not divisible by $k$, plays a key role in the work of Lovejoy \cite{L}. 

As part of his work, Andrews \cite{A_sing} uses elementary generating function manipulations to prove that, for all $n\geq 0,$ 
\begin{equation}
\label{And_congs_mod3}
\overline{C}_{3,1}(9n+3) \equiv \overline{C}_{3,1}(9n+6) \equiv 0 \pmod{3}.  
\end{equation} 
In Section \ref{C31_section}, we prove (\ref{And_congs_mod3}) as part of an infinite family of mod 3 congruences satisfied by $\overline{C}_{3,1}(n).$  
We also prove a number of arithmetic properties modulo powers of 2 satisfied by $\overline{C}_{3,1}(n).$   In Section \ref{C41_section}, we prove similar results for $\overline{C}_{4,1}(n)$ while in Section \ref{C6_section}, we prove a wide variety of results for $\overline{C}_{6,1}(n)$ and $\overline{C}_{6,2}(n),$ respectively.  All of the proofs will follow from elementary generating function considerations and $q$--series manipulations.  

Before we transition to our proofs, we note that, for $k\geq 3$ and $1\leq i\leq \lfloor \frac{k}{2}\rfloor,$ the generating function for $\overline{C}_{k,i}(n)$ is given by 
\begin{equation}
\label{main_genfn}
\sum_{n=0}^\infty \overline{C}_{k,i}(n)q^n = \frac{(q^k;q^k)_\infty(-q^i;q^k)_\infty(-q^{k-i};q^k)_\infty}{(q;q)_\infty}
\end{equation}
where 
$$
(A;q)_n = (1-A)(1-Aq)\dots (1-Aq^{n-1})
$$
and 
$$
(A;q)_\infty = \lim_{n\to\infty} (A;q)_n.
$$
For certain values of $k$ and $i,$ (\ref{main_genfn}) can be manipulated in elementary ways to generate the Ramanujan--like congruences which appear in this paper.

\section{Results for $\overline{C}_{3,1}(n)$}
\label{C31_section}

Motivated by Andrews, we first focus on the generating function for $\overline{C}_{3,1}(n)$ which, according to (\ref{main_genfn}), is given by 
\begin{eqnarray*}
\sum_{n=0}^\infty\overline{C}_{3,1}(n)q^n
&=& 
\frac{(q^3;q^3)_\infty(-q;q^3)_\infty(-q^2;q^3)_\infty}{(q;q)_\infty}\\
&=& 
\frac{(q^3;q^3)_\infty(-q;q)_\infty}{(q;q)_\infty(-q^3;q^3)_\infty}\\
&=& 
\frac{(q^3;q^3)_\infty^2(q^2;q^2)_\infty}{(q;q)_\infty^2(q^6;q^6)_\infty}\\
&=& 
\frac{(q^3;q^3)_\infty^2}{(q^6;q^6)_\infty}\Big{/}\frac{(q;q)_\infty^2}{(q^2;q^2)_\infty}\\
&=& \frac{\varphi(-q^3)}{\varphi(-q)},
\end{eqnarray*}
where $\varphi(q)$ is Ramanujan's theta function given by 
\begin{equation}
\label{phi_defn}
\varphi(q) = 1+2\sum_{n\geq 1}q^{n^2}.
\end{equation}
Given that
\begin{equation}
\label{genfn31phis}
\sum_{n=0}^\infty\overline{C}_{3,1}(n)q^n = \frac{\varphi(-q^3)}{\varphi(-q)},
\end{equation}
we can see rather quickly how one might develop congruences modulo 3 which are satisfied by $\overline{C}_{3,1}(n).$  

\begin{theorem}
\label{C31mod3general}
Let $N$ be a positive integer which is not expressible as the sum of two nonnegative squares.  Then 
$$
\overline{C}_{3,1}(N) \equiv 0 \pmod{3}.  
$$
\end{theorem}

\begin{proof}
Note that 
\begin{eqnarray*}
\sum_{n=0}^\infty(-1)^n\overline{C}_{3,1}(n)q^n
&=& 
\frac{\varphi(q^3)}{\varphi(q)}\\
&\equiv &
\frac{\varphi(q)^3}{\varphi(q)}\pmod3\\
&=& 
\varphi(q)^2\\
&=& 
\left (1+2\sum_{n=1}^\infty q^{n^2}\right )^2\\
&\equiv &
1+\sum_{n=1}^\infty q^{n^2}+\sum_{m,n=1}^\infty q^{m^2+n^2} \pmod{3}.
\end{eqnarray*}
The result follows.  
\end{proof}
Let $r_2(n)$ be the number of representations of $n$ as the sum of two squares. From the proof of Theorem \ref{C31mod3general}, we know
$$\overline{C}_{3,1}(n)\equiv(-1)^nr_2(n)\pmod3.$$
Recall the well--known formula for $r_2(n),$ as noted in \cite{by}, which states
$$r_2(n)=4\prod_{\substack{p|n\\p\equiv1\pmod4}}(1+\nu_p(n))\prod_{\substack{p|n\\p\equiv3\pmod4}}\frac{1+(-1)^{\nu_p(n)}}{2},$$
where $p$ is prime and $\nu_p(n)$ is the exponent of $p$ dividing $n$. In light of this formula for $r_2(n),$ we have the following corollaries of Theorem \ref{C31mod3general}.

\begin{corollary}
\label{cor1mod3}
For all $k,m\geq 0,$ 
$$
\overline{C}_{3,1}(2^k(4m+3)) \equiv 0 \pmod{3}.
$$
\end{corollary}

\begin{corollary}
Let $p\equiv1\pmod4$ be prime. Then for all  $k,m\ge0$ with $p\nmid m,$
$$\overline{C}_{3,1}(p^{3k+2}m)\equiv0\pmod3.$$
\end{corollary}

\begin{corollary}
\label{cor2mod3}
Let $p\equiv3\pmod4$ be prime. Then for all $k,m\ge0$ with $p\nmid m,$
$$\overline{C}_{3,1}(p^{2k+1}m)\equiv0\pmod3.$$
\end{corollary}

\begin{remark}
Note that Andrews' original congruences modulo 3, as given in (\ref{And_congs_mod3}), are the $p = 3, k = 0$ and $m\equiv 1,2\pmod3$ cases  of Corollary \ref{cor2mod3}.
\end{remark}

We now transition to a consideration of congruence results satisfied by $\overline{C}_{3,1}(n)$ modulo small powers of 2.  We begin with a lemma which will allow us to obtain an alternate form of the generating function for $\overline{C}_{3,1}(n)$ from which we can obtain such results.  

\begin{lemma}
$\varphi(-q^2)^2 = \varphi(q)\varphi(-q)$ where $\varphi(q)$ is defined in (\ref{phi_defn}).  
\end{lemma}

\begin{proof}
$$
\varphi(q)\varphi(-q)=\frac{(q^2;q^2)_\infty^5}{(q;q)_\infty^2(q^4;q^4)_\infty^2}\cdot\frac{(q;q)_\infty^2}{(q^2;q^2)_\infty}
=\frac{(q^2;q^2)_\infty^4}{(q^4;q^4)_\infty^2}=\varphi(-q^2)^2.
$$
\end{proof}

\begin{corollary}
\label{corA}
$$
\frac{1}{\varphi(-q)} = \frac{\varphi(q)}{\varphi(-q^2)^2}.
$$
\end{corollary}

\begin{proof}
This result is obvious based on the previous lemma.  
\end{proof}

\begin{corollary}
\label{alternate_genfn31}
$$
\sum_{n= 0}^\infty \overline{C}_{3,1}(n)q^n = \varphi(-q^3)\prod_{i=0}^\infty \varphi(q^{2^i})^{2^i}.
$$
\end{corollary}

\begin{proof}
We simply use (\ref{genfn31phis}) and iterate Corollary \ref{corA} ad infinitum.  
\end{proof}
We can now state a few characterization theorems for $\overline{C}_{3,1}(n)$ modulo small powers of 2.  

\begin{theorem}
For all $n\geq 1,$ $\overline{C}_{3,1}(n) \equiv 0\pmod{2}.$  
\end{theorem}
\begin{proof}
This follows from (\ref{genfn31phis}) and (\ref{phi_defn}).
\end{proof}

\begin{theorem}
\label{C31mod4}
For all $n\geq 1,$ 
\begin{equation*}
\overline{C}_{3,1}(n) \equiv 
\begin{cases}
2\pmod{4} & \text{if } n = k^2 \text{ or } n = 3k^2,\\
0\pmod{4} & \text{otherwise. }
\end{cases}
\end{equation*}
\end{theorem}

\begin{proof}
Thanks to Corollary \ref{alternate_genfn31} and (\ref{phi_defn}) we know
\begin{eqnarray*}
\sum_{n= 0}^\infty \overline{C}_{3,1}(n) q^n
&=& 
\varphi(-q^3)\prod_{i=0}^\infty \varphi(q^{2^i})^{2^i} \\
&\equiv &
\varphi(-q^3)\varphi(q) \pmod{4} \\
&=& 
\left( 1+2\sum_{k\geq 1}(-1)^kq^{3k^2} \right)\left( 1+2\sum_{k\geq 1}q^{k^2} \right) \\
&\equiv & 
 1+2\sum_{k\geq 1}q^{k^2}+2\sum_{k\geq 1}q^{3k^2} \pmod{4}.
\end{eqnarray*}
The result follows.  
\end{proof}
It is clear that one can write down numerous Ramanujan--like congruences modulo 4 satisfied by $\overline{C}_{3,1}(n)$ thanks to Theorem \ref{C31mod4}.  We refrain from doing so here.  

Note that it is also possible to write a relatively clean characterization modulo 8 for $\overline{C}_{3,1}(n)$ using the same strategy as that employed above.  This is because 
\begin{eqnarray*}
\sum_{n= 0}^\infty \overline{C}_{3,1}(n) q^n
&=& 
\varphi(-q^3)\prod_{i=0}^\infty \varphi(q^{2^i})^{2^i} \\
&\equiv &
\varphi(-q^3)\varphi(q)\varphi(q^2)^2 \pmod{8}.
\end{eqnarray*}
With that said, we consider a slight variant, namely obtaining a clean characterization modulo 8 for $(-1)^n\overline{C}_{3,1}(n),$ which can then be used rather quickly to prove numerous Ramanujan--like congruences modulo 8. 

\begin{theorem}
\label{C31mod8}
\begin{eqnarray*}
&&\sum_{n=0}^\infty (-1)^n\overline{C}_{3,1}(n)q^n \\
&\equiv & 
 1+6\sum_{k\geq 1}q^{k^2} + 4\sum_{k \geq 1}q^{2k^2} + 2\sum_{k\geq 1}q^{3k^2}  + 4\sum_{k,\ell \geq 1}q^{k^2+3\ell^2}  \pmod{8}.
\end{eqnarray*}
\end{theorem}

\begin{proof}
Since $\displaystyle{\varphi(q) = 1+2\sum_{n\geq 1}q^{n^2}},$ it is clear that $\varphi(q)^4 \equiv 1 \pmod{8}.$  It follows that 
\begin{eqnarray*}
&& 
\sum_{n=0}^\infty (-1)^n\overline{C}_{3,1}(n)q^n \\
&=&
\frac{\varphi(q^3)}{\varphi(q)} \\
&\equiv & 
\varphi(q^3)\varphi(q)^3 \pmod{8} \\
&=& 
\left(  1+2\sum_{k\geq 1}q^{k^2} \right)^3 \left(  1+2\sum_{k\geq 1}q^{3k^2} \right) \\
&\equiv & 
\left(  1+6\sum_{k\geq 1}q^{k^2}  + 4\sum_{k,\ell \geq 1}q^{k^2+\ell^2} \right) \left(  1+2\sum_{k\geq 1}q^{3k^2} \right) \pmod{8} \\
&\equiv & 
 1+6\sum_{k\geq 1}q^{k^2} + 2\sum_{k\geq 1}q^{3k^2} + 4\sum_{k,\ell \geq 1}q^{k^2+\ell^2} + 4\sum_{k,\ell \geq 1}q^{k^2+3\ell^2}  \pmod{8} \\
 &\equiv & 
 1+6\sum_{k\geq 1}q^{k^2} + 4\sum_{k \geq 1}q^{2k^2} + 2\sum_{k\geq 1}q^{3k^2}  + 4\sum_{k,\ell \geq 1}q^{k^2+3\ell^2}  \pmod{8}
\end{eqnarray*}
since solutions of $n=k^2+\ell^2$ with $k\not=\ell$ come in pairs.
\end{proof}

We close this section by briefly noting a few corollaries of Theorem \ref{C31mod8}.  

\begin{corollary}
For all $k, m\geq 0,$ 
\begin{eqnarray*}
\overline{C}_{3,1}(4^k(16m+6)) &\equiv & 0 \pmod{8}, \\
\overline{C}_{3,1}(4^k(16m+10)) &\equiv & 0 \pmod{8}, \text{\ \ and}\\
\overline{C}_{3,1}(4^k(16m+14)) &\equiv & 0 \pmod{8}.
\end{eqnarray*}
\end{corollary}

\begin{corollary}
For all $k, m\geq 0,$ 
\begin{eqnarray*}
\overline{C}_{3,1}(2^k(6m+5)) &\equiv & 0 \pmod{8}.
\end{eqnarray*}
\end{corollary}

\begin{corollary}
Let $p$ be prime, $p\equiv 5, 11\pmod{12}.$  For all $k, m\geq 0$ with $p\nmid m,$  $$\overline{C}_{3,1}(p^{2k+1}m) \equiv 0 \pmod{8}.$$  
\end{corollary}

\begin{proof}
Note that $n=p^{2k+1}m$ with $p\nmid m$ is neither a square, twice a square, three times a square, nor of the form $x^2+3y^2$
(since $\left(\frac{-3}{p}\right)=-1$ and $\nu_p(n)$ is odd).
\end{proof}

\section{Results for $\overline{C}_{4,1}(n)$}
\label{C41_section}
We now wish to consider other examples of the functions $\overline{C}_{k,i}(n)$ where arithmetic properties can be proven using elementary means.  
In this section, we concentrate on the function $\overline{C}_{4,1}(n).$  

\begin{theorem}
\label{genfn_C41}
\begin{equation*}
\sum_{n= 0}^\infty \overline{C}_{4,1}(n)q^n  = \frac{(q^2;q^2)_\infty^2}{(q;q)_\infty^2}.
\end{equation*}
\end{theorem}

\begin{proof}
Beginning with (\ref{main_genfn}), we see that 
\begin{eqnarray*}
\sum_{n= 0}^\infty \overline{C}_{4,1}(n)q^n 
&=& 
\frac{(q^4;q^4)_\infty(-q;q^4)_\infty(-q^3;q^4)_\infty}{(q;q)_\infty} \\
&=& 
\frac{(q^4;q^4)_\infty(-q;q^2)_\infty}{(q;q)_\infty} \\
&=& 
\frac{(q^4;q^4)_\infty(q^2;q^4)_\infty}{(q;q)_\infty(q;q^2)_\infty} \\
&=& 
\frac{(q^2;q^2)_\infty^2}{(q;q)_\infty^2}.
\end{eqnarray*}
\end{proof}
Theorem \ref{genfn_C41} provides the following characterization of $\overline{C}_{4,1}(n)$ modulo 2.  

\begin{theorem}
\label{C41mod2}
For all $n\geq 1,$ 
\begin{equation*}
\overline{C}_{4,1}(n) \equiv 
\begin{cases}
1\pmod{2} & \text{if } n = k(3k-1),\\
0\pmod{2} & \text{otherwise. }
\end{cases}
\end{equation*}
\end{theorem}

\begin{proof}
Thanks to Theorem \ref{genfn_C41}, we have 
\begin{eqnarray*}
\sum_{n= 0}^\infty \overline{C}_{4,1}(n) q^n
&=& 
\frac{(q^2;q^2)_\infty^2}{(q;q)_\infty^2}\\
&\equiv &
\frac{(q^2;q^2)_\infty^2}{(q^2;q^2)_\infty} \pmod{2}\\
&=& 
(q^2;q^2)_\infty \\
&\equiv &
\sum_{k=-\infty}^\infty q^{k(3k-1)} \pmod{2}
\end{eqnarray*}
thanks to Euler's Pentagonal Number Theorem \cite[Corollary 1.7]{A}.  
\end{proof}
Two corollaries follow immediately from the above.  

\begin{corollary}
For all $n\geq 0,$ 
$$\overline{C}_{4,1}(2n+1) \equiv 0\pmod{2}.$$
\end{corollary}
\begin{proof}
Note that $k(3k-1)$ is even for all integers $k.$  
\end{proof}
\begin{corollary}
Let $p$ be prime and let $1\leq r\leq p-1$ with $12r+1$ a quadratic nonresidue modulo $p.$  Then, for all $m\geq 0,$ 
$$
\overline{C}_{4,1}(pm+r) \equiv 0\pmod{2}.
$$
\end{corollary}
\begin{proof}
We have
$$
\sum_{n=0}^\infty\overline{C}_{4,1}(n)q^{12n+1}\equiv\sum_{k= -\infty}^\infty q^{(6k-1)^2}\pmod2.
$$
Here $n=pm+r$, so $12n+1=12pm+12r+1\equiv12r+1\pmod{p}$ is not a square modulo $p.$  Thus, $12n+1$ is
not a square, and $\overline{C}_{4,1}(n)\equiv0\pmod2$.
\end{proof}
From Theorem \ref{genfn_C41}, we can also obtain results modulo 4 satisfied by $\overline{C}_{4,1}(n).$  

\begin{theorem}
\label{C41mod4}
If $n$ cannot be represented as the sum of two pentagonal numbers, or if $n$ cannot be represented as the sum of a square and four times a pentagonal number,  then $\overline{C}_{4,1}(n) \equiv 0 \pmod{4}.$   
\end{theorem}

\begin{proof}
From the fact that $(1-q^2)^2\equiv(1-q)^4\pmod4$, we find
\begin{eqnarray*}
\sum_{n=0}^\infty\overline{C}_{4,1}(n)q^n 
&=& 
\frac{(q^2;q^2)_\infty^2}{(q;q)_\infty^2}\\
&\equiv &
\frac{(q;q)_\infty^4}{(q;q)_\infty^2} \pmod{4}\\
&=&
(q;q)_\infty^2\\
&=&\sum_{k,l=-\infty}^\infty (-1)^{k+l}q^{\frac{k(3k-1)}{2}+\frac{l(3l-1)}{2}},
\end{eqnarray*}
and
\begin{eqnarray*}
\sum_{n=0}^\infty(-1)^n\overline{C}_{4,1}(n)q^n
&\equiv &
(-q;q^2)_\infty^2(q^2;q^2)_\infty^2 \pmod{4}\\
&=&
\frac{(q^2;q^4)_\infty^2(q^2;q^2)_\infty^2}{(q;q^2)_\infty^2}\\
&=&
\frac{(q^2;q^2)_\infty^6}{(q;q)_\infty^2(q^4;q^4)_\infty^2}\\
&=& 
(q^4;q^4)_\infty\varphi(q)\\
&=&
\sum_{k,l=-\infty}^\infty(-1)^kq^{2k(3k-1)+\ell^2}.
\end{eqnarray*}
The result follows.  
\end{proof}
Using Theorem \ref{C41mod4}, we can explicitly write infinitely many Ramanujan--like congruences modulo 4 satisfied by $\overline{C}_{4,1}(n).$  

\begin{corollary}
\label{C41mod4primes3}
Let $p\ge5$ be a prime and $p\not\equiv1\pmod{12}.$ Then for all $k, m\ge0$ with $p\nmid m$,
$$\overline{C}_{4,1}\left (p^{2k+1}m+\frac{p^{2k+2}-1}{12}\right )\equiv0\pmod4.$$
\end{corollary}

\begin{proof}
First suppose $p$ is prime, $p\equiv 7$ or $11\pmod{12}.$
We have
$$
\sum_{n=0}^\infty\overline{C}_{4,1}(n)q^{24n+2}
\equiv
\sum_{k,l=-\infty}^\infty (-1)^{k+l}q^{(6k-1)^2+(6l-1)^2}\pmod4.
$$
Thus,  if $24n+2$ is not the sum of two squares, $\overline{C}_{4,1}(n)\equiv0\pmod4$.

We have $n=p^{2k+1}m+\displaystyle\frac{p^{2k+2}-1}{12}$, so
$$24n+2=24p^{2k+1}m+2p^{2k+2}=p^{2k+1}\left (24m+2p\right )$$ 
and $\nu_p(24n+2)$ is odd.
By Fermat's two--squares theorem, $24n+2$ is not the sum of two squares, so 
$\overline{C}_{4,1}(n)\equiv0\pmod4$.

Now suppose $p$ is prime, $p\equiv 5\pmod{12}.$
We have 
$$
\sum_{n=0}^\infty(-1)^n\overline{C}_{4,1}(n)q^{12n+1}\equiv
\sum_{k,l=-\infty}^\infty(-1)^kq^{(6k-1)^2+3(2l)^2} \pmod{4}.
$$
If $N$ is of the form $x^2+3y^2$,   
then it follows
by a standard argument that $\nu_p(N)$ is even since $\displaystyle{\left (\frac{-3}{p}\right )=-1}.$

However, here $n=p^{2k+1}m+\displaystyle\frac{p^{2k+2}-1}{12}$ and $\nu_p(12n+1)$ is odd.
So $12n+1$ is not of the form $x^2+3y^2$,   and $\overline{C}_{4,1}(n)\equiv0\pmod4.$
\end{proof}

As we close this section, we note that the generating function for $\overline{C}_{4,1}(n)$ is a modular function on $\Gamma_0(2).$  As such, we can slightly modify the proof of Theorem 1 of \cite{SCC} to obtain the following: 

\begin{theorem}
Let $p\geq 5$ be prime and let $\delta_p$ be the least positive residue of $p$ modulo 12.  Then, for all $m\geq 0$ with $p\nmid m,$ 
$$
\overline{C}_{4,1}\left(pm + \frac{p^2-1}{12}\right) \equiv 0\pmod{2^{\delta_p - 1}}.
$$
\end{theorem}
Hence, for example, we have the following:  

\begin{corollary}
For all $m\geq 0,$ 
\begin{eqnarray*}
\overline{C}_{4,1}(5m+2) &\equiv & 0\pmod{2^4} \text{\ \ if\ \  } 5\nmid m, \\
\overline{C}_{4,1}(7m+4) &\equiv & 0\pmod{2^6} \text{\ \ if\ \  } 7\nmid m, \\
\overline{C}_{4,1}(11m+10) &\equiv & 0\pmod{2^{10}} \text{\ \ if\ \  } 11\nmid m.
\end{eqnarray*}

\end{corollary}

\section{Results for $\overline{C}_{6,1}(n)$ and $\overline{C}_{6,2}(n)$}
\label{C6_section}
Next, we consider the two functions $\overline{C}_{6,1}(n)$ and $\overline{C}_{6,2}(n).$  We begin by proving an elementary result for $\overline{C}_{6,1}(n)$ modulo 3. 

\begin{theorem}
\label{C61mod3}
If $n$ cannot be represented as the sum of a pentagonal number and twice a triangular 
number, or if $n$ cannot be represented as the sum of a triangular number and four times a pentagonal number, then
$\overline{C}_{6,1}(n)\equiv0\pmod3$.
\end{theorem}

\begin{proof}
Beginning with (\ref{main_genfn}), we see that 
\begin{eqnarray*}
\sum_{n=0}^\infty\overline{C}_{6,1}(n)q^n
&=&
\frac{(q^6;q^6)_\infty(-q;q^6)_\infty(-q^5;q^6)_\infty}{(q;q)_\infty}\\
&=&
\frac{(q^6;q^6)_\infty(-q;q^2)_\infty}{(q;q)_\infty(-q^3;q^6)_\infty}\\
&=&
\frac{(q^6;q^6)_\infty(q^2;q^4)_\infty(q^3;q^6)_\infty}{(q;q)_\infty(q;q^2)_\infty(q^6;q^{12})_\infty}\\
&=&
\frac{(q^2;q^2)_\infty^2(q^3;q^3)_\infty(q^{12};q^{12})_\infty}
{(q;q)_\infty^2(q^4;q^4)_\infty(q^6;q^6)_\infty}\\
&\equiv &
\frac{(q^2;q^2)_\infty^2(q;q)_\infty^3(q^4;q^4)_\infty^3}{(q;q)_\infty^2(q^4;q^4)_\infty(q^2;q^2)_\infty^3}\pmod3\\
&=&
\frac{(q;q)_\infty(q^4;q^4)_\infty^2}{(q^2;q^2)_\infty}\\
&=& 
(q;q)_\infty\psi(q^2)\\
&=& 
\sum_{k=-\infty}^\infty\sum_{\ell=1}^\infty (-1)^kq^{\frac{k(3k-1)}{2}+\ell(\ell-1)},
\end{eqnarray*}
and 
\begin{eqnarray*}
\sum_{n=0}^\infty (-1)^n\overline{C}_{6,1}(n)q^n
&\equiv &
\frac{(-q;q^2)_\infty(q^2;q^2)_\infty(q^4;q^4)_\infty^2}{(q^2;q^2)_\infty}\pmod{3}\\
&=& 
\frac{(q^2;q^4)_\infty(q^4;q^4)_\infty^2}{(q;q^2)_\infty}\\
&=&
\frac{(q^2;q^2)_\infty^2(q^4;q^4)_\infty}{(q;q)_\infty}\\
&=& 
(q^4;q^4)_\infty\psi(q)\\
&=&
\sum_{k=-\infty}^\infty\sum_{\ell=1}^\infty (-1)^kq^{2k(3k-1)+\frac{\ell(\ell-1)}{2}}.
\end{eqnarray*}
Here we have used Ramanujan's theta function $\psi(q)$ which satisfies 
$$\psi(q)=\sum_{\ell=1}^\infty q^{\frac{\ell(\ell-1)}{2}}=\frac{(q^2;q^2)^2_\infty}{(q;q)_\infty}.$$
The result follows.  
\end{proof}

From Theorem \ref{C61mod3}, 
we can obtain the following  set of Ramanujan--like congruences modulo 3 for $\overline{C}_{6,1}(n).$

\begin{corollary}
\label{corC61mod3}
Let $p\geq 5$ be prime and $p\not\equiv 1,\ 7 \pmod{24}.$  Then for $k, m\ge0$ with $p\nmid m$,
$$
\overline{C}_{6,1}\left (p^{2k+1}m+7\times\frac{p^{2k+2}-1}{24}\right )\equiv0\pmod3.
$$
\end{corollary}
\begin{proof}
First suppose $p$ is prime, $p\equiv13,\ 17,\ 19$ or $23\pmod{24}.$
We have
$$
\sum_{n=0}^\infty\overline{C}_{6,1}(n)q^{24n+7}\equiv\sum_{k=-\infty}^\infty\sum_{\ell=1}^\infty
q^{(6k-1)^2+6(2\ell-1)^2}\pmod3.
$$
So, if $24n+7$ is not of the form $(6k-1)^2+6(2\ell-1)^2$, then $\overline{C}_{6,1}(n)\equiv0\pmod3$.

If $N$ is of the form $x^2+6y^2$, then $\nu_p(N)$ is even since $\displaystyle\left (\frac{-6}{p}\right )=-1.$
However, here
$
n=p^{2k+1}m+7\times\displaystyle\frac{p^{2k+2}-1}{24},
$
and
$\nu_p(24n+7)$ is odd. 
So $24n+7$ is not of the form $x^2+6y^2$, and $\overline{C}_{6,1}(n)\equiv0\pmod3$.

Now suppose $p$ is prime, $p\equiv 5$ or $11\pmod{24}.$ 
We have
$$
\sum_{n\ge0}(-1)^n\overline{C}_{6,1}(n)q^{24n+7}\equiv\sum_{k=-\infty}^\infty\sum_{\ell=1}^\infty (-1)^kq^{(12k-2)^2+3(2\ell-1)^2}.
$$
If $N$ is of the form $x^2+3y^2$, 
then $\nu_p(N)$ is even since $\displaystyle\left (\frac{-3}{p}\right )=-1$. However, $\nu_p(24n+7)$ is odd, so $24n+7$ is not of the form 
$x^2+3y^2.$ Therefore, we can conclude that $\overline{C}_{6,1}(n)\equiv0\pmod3$.
\end{proof}

We now transition to a similar analysis of $\overline{C}_{6,2}(n).$  

\begin{theorem}
\label{C62mod2}
For all $n\geq 1,$ 
\begin{equation*}
\overline{C}_{6,2}(n) \equiv 
\begin{cases}
1\pmod{2} & \text{if } n \text{\ is a pentagonal number},\\
0\pmod{2} & \text{otherwise. }
\end{cases}
\end{equation*}
\end{theorem}

\begin{proof}
Beginning with (\ref{main_genfn}), we have 
\begin{eqnarray*}
\sum_{n=0}^\infty\overline{C}_{6,2}(n)q^n
&=&
\frac{(q^6;q^6)_\infty(-q^2;q^6)_\infty(-q^4;q^6)_\infty}{(q;q)_\infty}\\
&=&
\frac{(q^6;q^6)_\infty(-q^2;q^2)_\infty}{(q;q)_\infty(-q^6;q^6)_\infty}\\
&=& 
\frac{(q^6;q^6)_\infty(q^4;q^4)_\infty(q^6;q^6)_\infty}{(q;q)_\infty(q^2;q^2)_\infty(q^{12};q^{12})_\infty}\\
&=& 
\frac{(q^4;q^4)_\infty(q^6;q^6)_\infty^2}{(q;q)_\infty(q^2;q^2)_\infty(q^{12};q^{12})_\infty}\\
&\equiv &
\frac{(q;q)_\infty^4(q^{12};q^{12})_\infty}{(q;q)_\infty(q;q)_\infty^2(q^{12};q^{12})_\infty}\pmod2\\
&=& 
(q;q)_\infty\\
&\equiv & 
\sum_{k= -\infty}^\infty q^{\frac{k(3k-1)}{2}}\pmod2.
\end{eqnarray*}
\end{proof}

\begin{corollary}
Let $p\geq 5$ be prime and let $1\leq r\leq p-1$ with $24r+1$ a quadratic nonresidue modulo $p.$  Then, for all $m\geq 0,$ 
$$
\overline{C}_{6,2}(pm+r) \equiv 0\pmod{2}.
$$
\end{corollary}

\begin{proof}
We have
$$
\sum_{n=0}^\infty\overline{C}_{6,2}(n)q^{24n+1}\equiv\sum_{k= -\infty}^\infty q^{(6k-1)^2}.
$$
Here, $n=pm+r$, so $24n+1=24pm+24r+1\equiv24r+1\pmod{p}$ is not a square modulo $p.$  Thus, $24n+1$ is not a square, and $\overline{C}_{6,2}(n)\equiv0\pmod2.$
\end{proof}

We close this section by considering $\overline{C}_{6,2}(n)$ modulo 3.  

\begin{corollary}
\label{C62mod3}
If $n$ cannot be represented as the sum of a pentagonal number and a square, or if $n$ cannot be written as the sum of a pentagonal number and twice a square, then
$$\overline{C}_{6,2}(n)\equiv0\pmod3.$$
\end{corollary}

\begin{proof}
We have
\begin{eqnarray*}
\sum_{n=0}^\infty\overline{C}_{6,2}(n)q^n
&=& 
\frac{(q^4;q^4)_\infty(q^6;q^6)_\infty^2}{(q;q)_\infty(q^2;q^2)_\infty(q^{12};q^{12})_\infty}\\
&\equiv & 
\frac{(q^4;q^4)_\infty(q^2;q^2)_\infty^6}{(q;q)_\infty(q^2;q^2)_\infty(q^4;q^4)_\infty^3}\pmod3\\
&=& 
\frac{(q^2;q^2)_\infty^5}{(q;q)_\infty(q^4;q^4)_\infty^2}\\
&=& 
(q;q)_\infty\varphi(q)\\
&=& 
\sum_{k,\ell=-\infty}^\infty (-1)^kq^{\frac{k(3k-1)}{2}+\ell^2}
\end{eqnarray*}
and
\begin{eqnarray*}
\sum_{n=0}^\infty(-1)^n\overline{C}_{6,2}(n)q^n
&\equiv & 
(-q;q^2)_\infty(q^2;q^2)_\infty\varphi(-q) \pmod{3}\\
&=& 
(-q;q^2)_\infty(q^2;q^2)_\infty\frac{(q;q)_\infty^2}{(q^2;q^2)_\infty}\\
&=& 
\frac{(q;q)_\infty^2(q^2;q^4)_\infty}{(q;q^2)_\infty}\\
&=& 
\frac{(q;q)_\infty^2(q^2;q^2)_\infty^2}{(q;q)_\infty(q^4;q^4)_\infty}\\
&=& 
\frac{(q;q)_\infty(q^2;q^2)_\infty^2}{(q^4;q^4)_\infty}\\
&=& 
(q;q)_\infty\varphi(-q^2)\\
&=& 
\sum_{k,\ell=-\infty}^\infty (-1)^l q^{\frac{k(3k-1)}{2}+2\ell^2}.
\end{eqnarray*}
The result follows.  
\end{proof}
We close our paper by demonstrating an infinite family of Ramanujan--like congruences satisfied by $\overline{C}_{6,2}(n)$ modulo 3.  
\begin{corollary}
\label{corC62mod3}
Let $p\geq 5$ be prime, $p\not\equiv1$ or $7\pmod{24}$, then for all $k,m\ge0$ with $p\nmid m$,
$$
\overline{C}_{6,2}\left (p^{2k+1}m+\frac{p^{2k+2}-1}{24}\right )\equiv0\pmod3.
$$
\end{corollary}

\begin{proof}
First suppose that $p$ is prime, $p\equiv 13,\ 17,\ 19$ or $23\pmod{24}.$  This means $\displaystyle\left (\frac{-6}{p}\right )=-1$.
We have
$$
\sum_{n=0}^\infty\overline{C}_{6,2}(n)q^{24n+1}\equiv\sum_{k,\ell=-\infty}^\infty(-1)^kq^{(6k-1)^2+6(2\ell)^2}.
$$
The proof now goes through just as the proof of the first half of Corollary \ref{corC61mod3}.

Next,  suppose $p$ is prime, $p\equiv 5$ or $11\pmod{24}.$ Then $\displaystyle\left (\frac{-3}{p}\right )=-1$.
We have
$$
\sum_{n=0}^\infty(-1)^n\overline{C}_{6,2}(n)q^{24n+1}\equiv\sum_{k,\ell=-\infty}^\infty(-1)^\ell
q^{(6k-1)^2+3(4\ell)^2}.
$$
The proof now goes through just as the proof of the second half of Corollary \ref{corC61mod3}.
\end{proof}

\bibliographystyle{amsplain}

\end{document}